\documentclass[12pt]{article}
\usepackage{amsmath,amsthm,amssymb,latexsym,color}
\usepackage{hyperref}

\theoremstyle{plain}
\newtheorem{theor}{Theorem}
\theoremstyle{remark}
\newtheorem{rem}{Remark}
\theoremstyle{plain}
\newtheorem{prop}[theor]{Proposition}

\newtheorem{lemma}[theor]{Lemma}

\def\R{{\mathbb R}}
\def\Prob{{\mathbb P}}

\def\dist{{\rm d}}

\def\Im{{\rm Im}}

\def\Vol{{\rm Vol}}

\def\Proj{{\rm P}}
\def\proj{{\rm p}}
\def\spn{{\rm span}}
\def\sign{{\rm sign}}

\def\Sph{{\rm S}}
\def\inter{{\rm int}}
\def\Class{{\mathcal C}}
\def\RandSet{{\mathcal S}}
\def\Event{{\mathcal E}}
\def\Ill{{\mathcal I}}
\def\Gauss{\nu}

\begin{document}

\title{Illumination of convex bodies with many symmetries}
\author{Konstantin Tikhomirov}

\maketitle

\begin{abstract}
Let $n\geq C$ for a large universal constant $C>0$, and let $B$ be a convex body in $\R^n$
such that for any $(x_1,x_2,\dots,x_n)\in B$, any choice of signs
$\varepsilon_1,\varepsilon_2,\dots,\varepsilon_n\in\{-1,1\}$ and
for any permutation $\sigma$ on $n$ elements we have
$(\varepsilon_1x_{\sigma(1)},\varepsilon_2x_{\sigma(2)},\dots,\varepsilon_nx_{\sigma(n)})\in B$.
We show that if $B$ is not a cube then $B$ can be illuminated
by strictly less than $2^n$ sources of light.
This confirms the Hadwiger--Gohberg--Markus illumination conjecture for unit balls
of $1$-symmetric norms in $\R^n$ for all sufficiently large $n$.
\end{abstract}

{\bf MCS (2010):} 52A20, 
52C17 (primary).


\section{Introduction}

Let $B$ be a convex body (i.e.\ a compact convex set with non-empty interior) in $\R^n$.
The well known problem of H.~Hadwiger \cite{H57}, independently formulated by
I.~Gohberg and A.~Markus, is to find the least number of smaller homothetic copies of $B$
sufficient to cover $B$. An equivalent question is to determine the smallest number $\Ill(B)$
of points in $\R^n\setminus B$ (``light sources'')
sufficient to illuminate $B$ \cite{H60}. Here, we say that a collection of points $\{p_1,p_2,\dots,p_m\}$ {\it illuminates} $B$
if for any point $x$ on the boundary of $B$ there is a point $p_i$ such that the line passing through $p_i$ and $x$
intersects the interior of $B$ at a point not between $p_i$ and $x$.
We refer to \cite[Chapter~VI]{BMS}, \cite[Chapter~3]{B10} and \cite{BK} for history of the question.

Following V.~Boltyanski (see, in particular, \cite[p.~256]{BMS}), we say that a boundary point $x\in B$
is {\it illuminated in a direction $y\in\R^n\setminus\{0\}$} if there is $\varepsilon>0$ such that
the point $x+\varepsilon y$ belongs to the interior of $B$. The entire body $B$ is illuminated in directions
$\{y^1,y^2,\dots,y^m\}$ if for every boundary point $x\in B$ there is $i\leq m$ such that
$x$ is illuminated in direction $y^i$. The smallest number of directions sufficient to illuminate $B$
is equal to $\Ill(B)$ (see \cite[Theorem~34.3]{BMS}).

It can be easily checked that the illumination number of an $n$-dimensional parallelotope is equal to $2^n$.
The Hadwiger--Gohberg--Markus illumination conjecture asserts that for any $n$-dimensional convex body $B$ different
from a parallelotope, $\Ill(B)<2^n$. We refer to \cite[Chapter~3]{B10} and \cite{BK}
for a list of results, confirming the conjecture in some special cases.
Here, let us mention a result of H.~Martini for so-called {\it belt polytopes} \cite{M}
and its extension to {\it belt bodies} due to V.~Boltyanski \cite{B96};
a paper of O.~Schramm \cite{S} dealing with bodies of constant width
and its generalization to {\it fat spindle bodies} by K.~Bezdek \cite{B12};
and a result of K.~Bezdek and T.~Bisztriczky \cite{BB} for dual cyclic polytopes.
For arbitrary convex bodies, the best known upper bound follows from C.A. Rogers' covering theorem \cite{R57}:
\begin{equation}\label{eq: rogers}
\Ill(B)\leq (n\log n+n\log\log n+5n)\frac{\Vol_n(B-B)}{\Vol_n(B)},
\end{equation}
where $B-B$ denotes the Minkowski difference and $\Vol_n(\cdot)$ is the standard Lebesgue volume in $\R^n$
(see, for example, \cite[Corollary~2.11]{BK} or \cite[Corollary~3.4.2]{B10}).

The aforementioned results of O.~Schramm \cite{S} and K.~Bezdek \cite{B12}
are based on a probabilistic argument in which directions of illumination are chosen
uniformly independently on the sphere $\Sph^{n-1}$ (see also \cite{Naszodi}).
Further, 
in a recent note \cite{LT}, it was shown that the general bound \eqref{eq: rogers}
can be recovered by illuminating the body $B$ with {\it independently distributed}
sources of light. It was suggested in \cite{LT} that randomized models of that type can
be helpful and may contribute towards solving the Illumination Problem.

In this note, we further develop the approach from \cite{S, B12, LT} by
applying it to convex bodies with many symmetries.
Let $n\geq 2$. We denote by $\Class_n$ the set of all convex bodies $B$ in $\R^n$ having the following properties:
\begin{enumerate}
\item[1)] $(\varepsilon_1 x_1,\varepsilon_2 x_2,\dots,\varepsilon_n x_n)\in B$ for any $(x_1,x_2,\dots,x_n)\in B$ and
any choice of signs $\varepsilon_i\in\{-1,1\}$ and
\item[2)] $(x_{\sigma(1)},x_{\sigma(2)},\dots,x_{\sigma(n)})\in B$ for any $(x_1,x_2,\dots,x_n)\in B$ and
any permutation $\sigma$ on $n$ elements.
\end{enumerate}

Note that the Minkowski functionals of convex bodies from $\Class_n$ are {\it $1$-symmetric
norms} in $\R^n$ (with respect to the standard basis),
and, conversely, the closed unit ball of any $1$-symmetric norm in $\R^n$ belongs to $\Class_n$.
The main result of this note is
\begin{theor}\label{t: main}
There is a universal constant $C>0$ with the following property: Let $n\geq C$ and let $B\in\Class_n$.
Assume that $B$ is not a cube. Then $\Ill(B)< 2^n$.
\end{theor}

\bigskip

Let us make some remarks. In the paper \cite{S} of O.~Schramm, it was proved that, given a group $G$
of orthogonal transformations of $\R^n$ which is generated
by reflections through hyperplanes and acts irreducibly on $\R^n$ (i.e.\ has no non-trivial invariant subspaces),
and a strictly convex body $B$ invariant under the action of $G$, we have $\Ill(B)=n+1$.
The group of orthogonal transformations generated by permutations of the standard basis
vectors and reflections with respect to coordinate hyperplanes, acts irreducibly on $\R^n$.
Hence, the result of O.~Schramm implies that for any strictly convex body $B\in\Class_n$ we have $\Ill(B)=n+1$.
However, the theorem of O.~Schramm gives no information about polytopes and, more generally,
bodies which are not strictly convex.

The proof of Theorem~\ref{t: main} is split into two parts. In the first part (Section~\ref{s: det}),
we illuminate bodies $B\in\Class_n$ with a small distance to the cube (to be defined below), using
purely deterministic arguments.
In the second part (Section~\ref{s: rand}), we construct a special set of random directions which
illuminate any given $B\in\Class_n$ with a ``large'' distance to the cube.

\section{Preliminaries}

Let us start with notation and basic definitions.
Given a finite set $I$, by $|I|$ we denote its cardinality. For any natural $k$,
we write $[k]$ instead of $\{1,2,\dots,k\}$. For a real number $r$, $\lfloor r\rfloor$
denotes the largest integer not exceeding $r$, and $\lceil r\rceil$ --- the smallest integer greater or equal to $r$.

Let $n$ be a natural number.
For a vector $x=(x_1,x_2,\dots,x_n)\in\R^n$, let
$$I_0^x:=\bigl\{i\leq n:\,x_i=0\bigr\}.$$
The standard basis vectors in $\R^n$ will be denoted by $e_1,e_2,\dots,e_n$ and the
standard inner product --- by $\langle\cdot,\cdot\rangle$. The maximum ($\ell_\infty^n$) norm
in $\R^n$ will be denoted by $\|\cdot\|_\infty$.
Given a convex body $B$ in $\R^n$, by $\partial B$ we denote its {\it boundary}, and by $\inter(B)$ its {\it interior}.
If $0\in \inter(B)$ then the {\it Minkowski functional} $\|\cdot\|_B$ on $\R^n$ is defined by
$$\|y\|_B:=\inf\bigl\{\lambda>0:\,y\in\lambda B\bigr\},\;\;\;y\in\R^n.$$

Further, for a convex body $B$ in $\R^n$ and a point $x\in\partial B$, let
{\it the Gauss image} $\Gauss(B,x)$ be the set of all outer normal unit vectors for supporting
hyperplanes at $x$. In other words, $\Gauss(B,x)$ is the set of all vectors $v\in\Sph^{n-1}$ such that
$\langle v,y-x\rangle\leq 0$ for all $y\in B$.
We omit the proof of the next lemma (see, for example, \cite[Lemma~4]{S} for an equivalent statement):
\begin{lemma}\label{l: trivial criterion}
Given a convex body $B$ in $\R^n$ ($n\geq 2$),
a direction $y\in\R^n\setminus\{0\}$ illuminates $x\in\partial B$ if and only if $\langle y,v\rangle<0$
for all $v\in \Gauss(B,x)$.
\end{lemma}

Let $n\geq 2$ and let the class $\Class_n$ be defined as in the Introduction.
It is easy to see that, given a body $B\in\Class_n$ and a vector
$(x_1,x_2,\dots,x_n)\in B$, we have $(\alpha_1 x_1,\alpha_2 x_2,\dots,\alpha_n x_n)\in B$
for any $\alpha_i\in[-1,1]$.
Hence, the following holds:
\begin{lemma}\label{l: always positive}
For any $B\in\Class_n$ ($n\geq 2$), any $(x_1,x_2,\dots,x_n)\in \partial B$ and $(v_1,v_2,\dots,v_n)\in \Gauss(B,x)$
we have $x_i v_i\geq 0$ for all $i\leq n$.
\end{lemma}
Again, the proof of Lemma~\ref{l: always positive} is straightforward, and we omit it.
\begin{lemma}\label{l: ordering}
Let $B\in\Class_n$ ($n\geq 2$) and let $x=(x_1,x_2,\dots,x_n)\in\partial B$.
Then for all $i,j\leq n$ such that $|x_i|>|x_j|$, we have $|v_i|\geq |v_j|$ for any $(v_1,v_2,\dots,v_n)\in \Gauss(B,x)$.
\end{lemma}
\begin{proof}
Assume the opposite: let $B\in\Class_n$, a vector $x=(x_1,x_2,\dots,x_n)\in\partial B$ and
$v=(v_1,v_2,\dots,v_n)\in \Gauss(B,x)$ be such that for some $i,j\leq n$ we have
$|x_i|>|x_j|$ and $|v_i|< |v_j|$.
Obviously,
$$H:=\bigl\{z\in\R^n:\,\langle z,v\rangle=\langle x,v\rangle\bigr\}$$
is a supporting hyperplane for $B$. Let $\varepsilon_i,\varepsilon_j\in\{-1,1\}$ be such that
$\varepsilon_i x_i v_j,\varepsilon_j x_j v_i\geq 0$, and denote
$$y:=\sum\limits_{k\neq i,j}x_ke_k+\varepsilon_i x_i e_j+\varepsilon_jx_j e_i.$$
Then
\begin{align*}
\langle y,v\rangle&=\langle x,v\rangle
+|x_i v_{j}|+|x_j v_i|-x_iv_i-x_jv_j\\
&=|\langle x,v\rangle|+(|x_i|-|x_j|)(|v_j|-|v_i|)\\
&>|\langle x,v\rangle|.
\end{align*}
Thus, $y$ cannot belong to $B$, contradicting the definition of the class $\Class_n$.
\end{proof}

Given two convex bodies $B$ and $\widetilde B$ in $\Class_n$, we define the {\it distance}
$\dist(B,\widetilde B)$ between $B$ and $\widetilde B$ as
$$\dist(B,\widetilde B)=\inf\bigl\{\lambda\geq 1:\,B\subset r\widetilde B\subset\lambda B\;\;\mbox{ for some }r>0\bigr\}.$$
In particular, $\dist(B,[-1,1]^n)$ is equal to the ratio $\|e_1+e_2+\dots+e_n\|_B/\|e_1\|_B$.
Note that $\dist(B,\widetilde B)$ is different from the {\it Banach--Mazur distance} between convex bodies.

\section{Illumination of convex bodies with a small distance to the cube}\label{s: det}

In this section, we consider the problem of illuminating a set $B\in\Class_n$
with a small distance to the cube.
Here, our construction is purely deterministic.
We prove the following:
\begin{prop}\label{p: det}
Let $B\in\Class_n$ ($n\geq 2$) with $1\neq \dist(B,[-1,1]^n)<2$. Then at least one of the following is true:
\begin{enumerate}
\item[1)] $B$ can be illuminated in directions
$$\bigl\{(\varepsilon_1,\varepsilon_2,\dots,\varepsilon_n)\in\{-1,1\}^n:\,\exists i\leq n-1\mbox{ with }\varepsilon_i=-1\bigr\}
\cup\{e_1+e_2+\dots+e_{n-1}\}.$$
\item[2)] $B$ can be illuminated in directions
$$\bigl(\{-1,1\}^{n-1}\times \{0\}\bigr)\cup \{\pm e_n\}.$$
\end{enumerate}
\end{prop}

Note that the first set in the above statement has cardinality $2^n-1$, and the second --- $2^{n-1}+2$.
The proposition is obtained as an easy corollary of Lemmas~\ref{l: T1} and~\ref{l: T2} given below.
But first, let us prove
\begin{lemma}\label{l: norm implication}
Let $B\in\Class_n$ ($n\geq 2$) and let $x=(x_1,x_2,\dots,x_n)\in\partial B$.
Further, let $y\in\{-1,0,1\}^n$ be a vector such that 1) $I_0^y\subset I_0^x$ and
2) for any $i\leq n$ such that $x_i\neq 0$, we have $y_i=-\sign(x_i)$.
Finally, assume that $x$ is not illuminated in the direction $y$. Then necessarily
$$\Bigl\|\sum_{i\in [n]\setminus I_0^y}e_i\Bigr\|_B\geq \frac{2}{\|x\|_\infty}.$$
\end{lemma}
\begin{proof}
In view of  Lemma~\ref{l: trivial criterion},
the fact that $y=(y_1,y_2,\dots,y_n)$ does not illuminate $x$ means that there is a vector
$v=(v_1,v_2,\dots,v_n)\in \Gauss(B,x)$ such that $\langle y,v\rangle\geq 0$. By the definition of $y$
and by Lemma~\ref{l: always positive}, we have
$$\sum_{i\in [n]\setminus I_0^x}y_i v_i=-\sum_{i\in [n]\setminus I_0^x}|v_i|.$$
Thus, the condition $\langle y,v\rangle\geq 0$ implies that
$$\sum_{i\in I_0^x\setminus I_0^y}y_i v_i\geq \sum_{i\in [n]\setminus I_0^x}|v_i|.$$
Clearly,
$$H:=\bigl\{z\in\R^n:\,\langle z,v\rangle=\langle x,v\rangle\bigr\}$$
is a supporting hyperplane for $B$. On the other hand, we have
$$\Bigl\langle \sum_{i\in [n]\setminus I_0^x}(-y_i) e_i+
\sum_{i\in I_0^x\setminus I_0^y}y_i e_i,v\Bigr\rangle
\geq 2\sum_{i\in [n]\setminus I_0^x}|v_i|\geq \frac{2\langle x,v\rangle}{\|x\|_\infty}.$$
Hence, the $\|\cdot\|_B$-norm of the vector $\sum_{i\in [n]\setminus I_0^x}(-y_i) e_i+
\sum_{i\in I_0^x\setminus I_0^y}y_i e_i$ is at least $2/\|x\|_\infty$.
The result follows.
\end{proof}

\begin{lemma}\label{l: T1}
Let $B\in\Class_n$ ($n\geq 2$) be such that $1\neq \dist(B,[-1,1]^n)<2$. Then at least one of the following is true:
\begin{enumerate}
\item[1)] $B$ can be illuminated in directions
$$T_1:=\bigl\{(\varepsilon_1,\varepsilon_2,\dots,\varepsilon_n)\in\{-1,1\}^n:\,\exists i\leq n-1\mbox{ with }\varepsilon_i=-1\bigr\}
\cup\{e_1+e_2+\dots+e_{n-1}\}.$$
\item[2)] $\|e_i+e_j\|_B>\|e_i\|_B$, $i\neq j$.
\end{enumerate}
\end{lemma}
\begin{proof}
Without loss of generality, we can assume that $\|e_i\|_B=1$ (note that this implies $B\subset [-1,1]^n$,
i.e.\ $\|\cdot\|_B\geq \|\cdot\|_\infty$).
Assume that the first condition is not satisfied.
Thus, there is a vector $x\in\partial B$ which is not illuminated in directions from $T_1$.
Consider three possibilities:
\begin{enumerate}
\item[a)] $I^x_0\neq \emptyset$.
Then we can find a vector $y\in T_1$ such that $I^y_0\subset I^x_0$ and $y_i=-\sign(x_i)$ for all $i\leq n$ with $x_i\neq 0$.
By Lemma~\ref{l: norm implication}, we have
$$\Bigl\|\sum_{i=1}^n e_i\Bigr\|_B\geq \frac{2}{\|x\|_\infty}\geq 2,$$
contradicting the assumption $\dist(B,[-1,1]^n)<2$.

\item[b)] $I^x_0= \emptyset$ and $|x_n|\leq |x_i|$ for all $i\leq n$.
We define $y=(y_1,y_2,\dots,y_n)$ by
$y_i:=-\sign(x_i)$ ($i\leq n-1$); $y_n:=0$ if $y_i=1$ for all $i\leq n-1$, or $y_n:=-\sign(x_n)$, otherwise.
It is not difficult to see that $y\in T_1$. Hence, direction $y$ does not illuminate $x$,
and, by Lemma~\ref{l: trivial criterion}, there is $v=(v_1,v_2,\dots,v_n)\in \Gauss(B,x)$ such that $\langle y,v\rangle\geq 0$.
In view of Lemma~\ref{l: always positive} and the definition of $y$,
this implies that $v_i=0$ for all $i\leq n-1$ (and $v_n=\pm 1$), whence
$$H:=\bigl\{z\in\R^n:\,\langle z,e_n\rangle=|x_n|\bigr\}$$
is a supporting hyperplane of $B$. On the other hand, $e_n\in B$ by our assumption, implying $|x_n|\geq 1$.
Thus, $|x_1|,|x_2|,\dots,|x_n|\geq 1$ and $x\in\partial B$.
But this contradicts the condition $B\subset [-1,1]^n$, $B\neq [-1,1]^n$.

\item[c)] $I^x_0= \emptyset$ and there is $j\leq n-1$ such that
$|x_j|\leq |x_i|$ for all $i\leq n$
(clearly, $j$ does not have to be unique). Define a vector $y=(y_1,y_2,\dots,y_n)$ by
$y_i:=-\sign(x_i)$ ($i\neq j$); $y_j:=-1$.
Again, $y\in T_1$. Hence,
there is $v=(v_1,v_2,\dots,v_n)\in \Gauss(B,x)$ such that $\langle y,v\rangle\geq 0$.
This implies, in view of Lemma~\ref{l: always positive},
\begin{equation}\label{eq: aux 165}
0\neq \sum_{i=1}^n|v_i|\leq 2|v_j|.
\end{equation}
On the other hand, in view of Lemma~\ref{l: ordering} we have $|v_j|\leq |v_i|$
for all $i\leq n$ such that $|x_i|>|x_j|$.
The last two conditions can be simultaneously fulfilled only if the set
$$J:=\bigl\{i\leq n:\,|x_i|>|x_j|\bigr\}$$
has cardinality at most $1$. The case $J=\emptyset$ (when all coordinates of $x$ are equal by absolute value)
was covered in part (b). Thus, we only need to consider the situation $|J|=1$.
Assume that $k\leq n$ is such that $|x_k|>|x_j|$. Then, by \eqref{eq: aux 165} and Lemma~\ref{l: ordering}, we have
$|v_k|= |v_j|$ and $v_i=0$ for all $i\neq k,j$.
Hence,
$$H:=\bigl\{z=(z_1,z_2,\dots,z_n)\in\R^n:\,z_k+z_j=|x_k|+|x_j|\bigr\}$$
is a supporting hyperplane for $B$. At the same time, $1=\|x\|_B\geq\|x_k e_k\|_B=|x_k|>|x_j|$, whence $|x_k|+|x_j|<2$.
This implies that $e_k+e_j\notin B$, i.e.\ $\|e_k+e_j\|_B>1$.
\end{enumerate}
\end{proof}

\begin{lemma}\label{l: T2}
Let $B\in\Class_n$ ($n\geq 2$) and assume that $\|e_i+e_j\|_B>\|e_i\|_B$, $i\neq j$.
Then $B$ can be illuminated in directions
$$T_2:=\bigl(\{-1,1\}^{n-1}\times \{0\}\bigr)\cup \{\pm e_n\}.$$
\end{lemma}
\begin{proof}
We will assume that $\|e_i\|_B=1$.
Let $x=(x_1,x_2,\dots,x_n)\in\partial B$. Consider two cases:
\begin{enumerate}
\item[a)] $|x_n|>|x_i|$ for all $i\leq n-1$. In view of Lemmas~\ref{l: always positive} and~\ref{l: ordering},
for any $v=(v_1,v_2,\dots,v_n)\in \Gauss(B,x)$ we have $v_n \neq 0$ and $\sign(v_n)=\sign(x_n)$. Hence
$x$ is illuminated by the vector $-\sign(x_n)e_n\in T_2$.

\item[b)] There is $j\leq n-1$ such that $|x_j|\geq |x_i|$ for all $i\leq n$.
Define $y=(y_1,y_2,\dots,y_n)$ as $y_i:=-\sign(x_i)$ for all $i\leq n-1$, and $y_n:=0$.
Obviously, $y\in T_2$.
If $y$ illuminates $x$ then we are done. Otherwise, by Lemmas~\ref{l: trivial criterion}
and~\ref{l: always positive}, for some $v=(v_1,v_2,\dots,v_n)\in \Gauss(B,x)$ we have
$$0\leq \langle y,v\rangle=-\sum\limits_{i=1}^{n-1}|v_i|.$$
Hence, $v=\pm e_n$, and the hyperplane
$$H:=\bigl\{z\in\R^n:\,\langle z,e_n\rangle=|x_n|\bigr\}$$
is supporting for $B$, whence $\|x_n e_n\|_B= 1$.
On the other hand, in view of the assumptions of the lemma, $\|x\|_B\geq \|x_je_j+x_ne_n\|_B>\|x_n e_n\|_B$.
We get that $\|x\|_B>1$, contradicting the choice of $x$.
\end{enumerate}
\end{proof}

\section{Randomized illumination of convex bodies far from the cube}\label{s: rand}

Assume that $n\geq 2$. Let $X$ be an $n$-dimensional random vector with i.i.d.\ coordinates taking values $+1$
and $-1$ with equal probability $1/2$. Further, let $\{X^\ell\}_{\ell=1}^{\infty}$ be copies of $X$.
Next, for any $m\leq n$ let $\Proj^{(m)}$ be the random coordinate projection in $\R^n$ of rank $m$,
such that the image of $\Proj^{(m)}$ is uniformly distributed on the set of all coordinate subspaces of dimension $m$.
In other words, for any sequence $i_1<i_2<\dots<i_m\leq n$ we have
$\Im\Proj^{(m)}=\spn\{e_{i_1},e_{i_2},\dots,e_{i_m}\}$ with probability ${n\choose m}^{-1}$.
Let also $\Proj^{(m)}_\ell$ ($\ell=1,2,\dots$) be copies of $\Proj^{(m)}$.
Additionally, we require that all the $X^\ell$ and $\Proj^{(m)}_\ell$ ($\ell=1,2,\dots$; $m\leq n$)
be jointly independent.
Now, for every $k\leq \lceil n/2\rceil$ we define a random (multi)set of vectors
\begin{equation}\label{eq: rset def}
\RandSet_k:=\bigl\{\Proj^{(2k-1)}_\ell (X^\ell)\bigr\}_{\ell=1}^{\lfloor 2^n/n^2\rfloor}.
\end{equation}
The cardinality $\lfloor 2^n/n^2\rfloor$ has no special meaning; we only need the condition
$$\Bigl|\bigcup_{k=1}^{\lceil n/2\rceil} \RandSet_{k}\Bigr|<2^{n-1},$$
together with the requirement that the individual sets $\RandSet_k$ are ``sufficiently large''.

\begin{lemma}\label{l: prob core}
There is a universal constant $C>0$ such that, given $n\geq C$ and any natural $k\leq \lceil n/2\rceil$,
the event
\begin{align*}
\Event_k:=\bigl\{&\mbox{For any }y=(y_1,y_2,\dots,y_n)\in\{-1,0,1\}^n\mbox{ with }|I^y_0|=n-k\mbox{ there is }\ell\leq 2^n/n^2\\
&\mbox{such that }\Proj^{(2k-1)}_\ell(y)=y\mbox{ and }X^\ell_i=y_i\mbox{ for all }i\in [n]\setminus I^y_0\bigr\}
\end{align*}
has probability at least $1-\exp(-2n)$. 
\end{lemma}
\begin{proof}
We shall assume that $n$ is large. Fix any natural $k\leq \lceil n/2\rceil$. Clearly, there are precisely
${n\choose k}2^k$ vectors in $\{-1,0,1\}^n$ whose supports have cardinality $k$. Hence, it is sufficient
to show that for any fixed $y\in\{-1,0,1\}^n$ with $|I^y_0|=n-k$, the probability of the event
$$\Event_y:=\bigl\{\mbox{There is }\ell\leq 2^n/n^2
\mbox{ such that }\Proj^{(2k-1)}_\ell(y)=y\mbox{ and }X^\ell_i=y_i\mbox{ for all }i\in [n]\setminus I^y_0\bigr\}$$
is at least $1-2^{-k}\exp(-2n){n\choose k}^{-1}$.

Take any $\ell\leq 2^n/n^2$. Obviously,
$$\Prob\bigl\{X^\ell_i=y_i\mbox{ for all }i\in [n]\setminus I^y_0\bigr\}=2^{-k}.$$
Next, in view of the definition of the projection $\Proj^{(2k-1)}_\ell$, we have
$$\Prob\bigl\{\Proj^{(2k-1)}_\ell(y)=y\bigr\}={n-k\choose k-1}{n\choose 2k-1}^{-1}.$$
Using Stirling's approximation, the last expression can be estimated as follows:
\begin{align*}
{n-k\choose k-1}{n\choose 2k-1}^{-1}&=\frac{(n-k)!(2k-1)!}{(k-1)!n!}\\
&\geq \frac{1}{n}\frac{(n-k)!(2k)!}{k!n!}\\
&\geq \frac{1}{2n}\frac{(n-k)^{n-k+1/2}(2k)^{2k+1/2}}{k^{k+1/2}n^{n+1/2}}\\
&\geq \frac{4^k}{2n^2}\Bigl(1-\frac{k}{n}\Bigr)^{n-k}\Bigl(\frac{k}{n}\Bigr)^{k}.
\end{align*}
Now, since $\Proj^{(2k-1)}_\ell$ and $X^\ell$ are independent, we get
$$\Prob\bigl\{\Proj^{(2k-1)}_\ell(y)=y\mbox{ and }X^\ell_i=y_i\mbox{ for all }i\in [n]\setminus I^y_0\bigr\}
\geq \frac{2^k}{2n^2}\Bigl(1-\frac{k}{n}\Bigr)^{n-k}\Bigl(\frac{k}{n}\Bigr)^{k}.$$
It is not difficult to check that the function $f(t):=2^t (1-t)^{1-t} t^t$, defined for $t\in[0,1]$,
takes its minimum at $t=1/3$. Hence,
$$\frac{2^k}{2n^2}\Bigl(1-\frac{k}{n}\Bigr)^{n-k}\Bigl(\frac{k}{n}\Bigr)^{k}
=\frac{1}{2n^2}f(k/n)^n\geq \frac{1}{2n^2}f(1/3)^n=\frac{1}{2n^2}\Bigl(\frac{2}{3}\Bigr)^n.$$

Finally, we get
\begin{align*}
1-\Prob(\Event_y)
&=\prod_{\ell=1}^{\lfloor 2^n/n^2\rfloor}
\Prob\bigl\{\Proj^{(2k-1)}_\ell(y)\neq y\mbox{ or }X^\ell_i\neq y_i\mbox{ for some }i\in [n]\setminus I^y_0\bigr\}\\
&\leq \Bigl(1-\frac{1}{2n^2}\Bigl(\frac{2}{3}\Bigr)^n\Bigr)^{\lfloor 2^n/n^2\rfloor}\\
&\ll 2^{-k}\exp(-2n){n\choose k}^{-1},
\end{align*}
provided that $n$ is sufficiently large. The result follows.
\end{proof}

Now, we can prove the following result which, together with Proposition~\ref{p: det},
gives the estimate $\Ill(B)<2^{n}$ for any $B\in\Class_n$ with $\dist(B,[-1,1]^n)\neq 1$.
\begin{prop}
There is a universal constant $C>0$ with the following property:
let $n\geq C$, $B\in\Class_n$, and assume that $\dist(B,[-1,1]^n)\geq 2$.
Define
$$T:=\{-1,1\}^{n-1}\times \{0\}.$$
Then with probability at least $1-\exp(-n)$ the set $B$ can be illuminated in directions
$$T\cup\bigcup_{k=1}^{\lceil n/2\rceil}\RandSet_{k},$$
where the random sets $\RandSet_{k}$ are defined by \eqref{eq: rset def}.
\end{prop}
\begin{proof}
Without loss of generality, we may assume that $\|e_i\|_B=1$.
First, we show that any vector $x\in\partial B$ with $|\{i\leq n:\,|x_i|=\|x\|_\infty\}|>\lceil n/2\rceil$
can be illuminated in a direction from $T$. Indeed, for any such vector $x$, since $\dist(B,[-1,1]^n)\geq 2$
and by the definition of the class $\Class_n$ and Lemma~\ref{l: ordering}, we necessarily have
$$\frac{\|x\|_B}{\|x\|_\infty}\geq \Bigl\|\sum_{i=1}^{\lceil n/2\rceil+1}e_i\Bigr\|_B
\geq \frac{1}{2}\bigl\|2e_1+e_2+e_3+\dots+e_n\bigr\|_B
>\frac{1}{2}\bigl\|e_1+e_2+\dots+e_n\bigr\|_B\geq 1.$$
So, $\|x\|_B>\|x\|_\infty$, whence for any $v\in \Gauss(B,x)$ we have $|I^v_0|\leq n-2$, and, in particular, $v\neq \pm e_n$.
Now, pick a vector $y=(y_1,y_2,\dots,y_n)\in T$ such that $y_i=-\sign(x_i)$ for
all $i\in [n-1]\setminus I^x_0$. For any $v\in \Gauss(B,x)$ we have
$$\langle y,v\rangle\leq -\sum_{i\in [n-1]\setminus I^x_0}|v_i|+\sum_{j\in I^x_0\setminus\{n\}}|v_j|.$$
By Lemma~\ref{l: ordering}, for any $i\in [n-1]\setminus I^x_0$ and $j\in I^x_0\setminus\{n\}$
we have $|v_i|\geq |v_j|$. Together with the obvious estimate $|[n-1]\setminus I^x_0|>|I^x_0\setminus\{n\}|$ and
the condition $v\neq \pm e_n$, this implies $\langle y,v\rangle<0$, i.e.\ $x$ is illuminated in direction $y$.

\medskip

Let events $\Event_k$ be defined as in Lemma~\ref{l: prob core}, and denote
$$\Event:=\bigcap_{k=1}^{\lceil n/2\rceil}\Event_k.$$
In view of Lemma~\ref{l: prob core}, $\Prob(\Event)\geq 1-\exp(-n)$,
provided that $n$ is sufficiently large.
For the rest of the proof, we fix realizations $x^\ell$ and
$\proj^{(2k-1)}_\ell$ of vectors $X^{\ell}$ and projections $\Proj^{(2k-1)}_\ell$, respectively,
($\ell=1,2,\dots$; $k\leq \lceil n/2\rceil$) from the event $\Event$.

Take any $x\in\partial B$ which is not illuminated in directions from $T$. By the above argument,
the set
$$J^x:=\bigl\{i\leq n:\,|x_i|=\|x\|_\infty\bigr\}$$
has cardinality at most $\lceil n/2\rceil$. Take $k:=|J^x|$. Then, applying the definition of $\Event$
to the vector $y:=-\sum_{i\in J^x}\sign(x_i)e_i$, we get that there is
$\ell\leq 2^n/n^2$ such that $\langle x^\ell,e_i\rangle=-\sign(x_i)$
for all $i\in J_x$ and the image of $\proj^{(2k-1)}_\ell$
contains $\spn\{e_i\}_{i\in J_x}$. Denote $\widetilde y:=\proj^{(2k-1)}_\ell(x^\ell)$.
We will show that $x$ is illuminated in direction $\widetilde y$.
Indeed, take any $v=(v_1,v_2,\dots,v_n)\in \Gauss(B,x)$.
Then
$$\langle \widetilde y,v\rangle\leq-\sum_{i\in J_x}|v_i|+\sum_{i\in [n]\setminus (J_x\cup I^{\widetilde y}_0)}|v_i|.$$
Note that by Lemma~\ref{l: ordering} we have $|v_i|\leq |v_j|$ for all $i\in [n]\setminus J_x$
and $j\in J_x$. Further, by the construction of $\widetilde y$ we have $|[n]\setminus (J_x\cup I^{\widetilde y}_0)|=k-1<|J_x|$.
Hence, $\langle \widetilde y,v\rangle$ is strictly negative. It remains to apply Lemma~\ref{l: trivial criterion}.

Thus, the convex body $B$ is illuminated by the union of directions $T\cup\bigcup_{k=1}^{\lceil n/2\rceil}\RandSet_{k}$
with probability at least $1-\exp(-n)$,
and the proof is complete.
\end{proof}

\begin{rem}
For the sake of keeping the presentation transparent, we did not attempt to compute the lower bound
for the dimension $n$ for which the proof starts to work. Neither did we try to decrease the
cardinality of the illuminating set. It is natural to ask whether the above argument can be
generalized to deal with ``$1$-unconditional'' bodies, i.e.\ convex bodies symmetric with respect to
coordinate hyperplanes. Unfortunately, our proof seems to use the permutation invariance in a crucial way,
and some essential new ingredients are needed.
\end{rem}

\noindent\small{Konstantin Tikhomirov,}\\
\small{University of Alberta, Canada}\\
\small{Current address: Princeton University, NJ}\\
\small{e-mail: kt12@math.princeton.edu}

\end{document}